\documentclass[12pt, leqno]{amsart}
\usepackage{amsmath}
\usepackage{amsfonts}
\usepackage{amssymb}
\usepackage{amsthm}
\usepackage{url}
\usepackage{graphicx} 

\newcommand{\R}{\mathbb{R}}

\DeclareMathOperator*{\osc}{osc}

\DeclareMathOperator*{\esssup}{ess\,sup}
\DeclareMathOperator*{\essinf}{ess\,inf}

\DeclareMathOperator{\Capc}{Cap}

\def\vint_#1{\mathchoice%
          {\mathop{\kern 0.2em\vrule width 0.6em height 0.69678ex depth -0.58065ex
                  \kern -0.8em \intop}\nolimits_{\kern -0.4em#1}}%
          {\mathop{\kern 0.1em\vrule width 0.5em height 0.69678ex depth -0.60387ex
                  \kern -0.6em \intop}\nolimits_{#1}}%
          {\mathop{\kern 0.1em\vrule width 0.5em height 0.69678ex depth -0.60387ex
                  \kern -0.6em \intop}\nolimits_{#1}}%
          {\mathop{\kern 0.1em\vrule width 0.5em height 0.69678ex depth -0.60387ex
                  \kern -0.6em \intop}\nolimits_{#1}}}

\newcommand{\art}[6]{{\sc #1, \rm #2, \it #3 \bf #4 \rm (#5), \mbox{#6}.}}
\newcommand{\book}[3]{{\sc #1, \it #2, \rm #3.}}
\newcommand{\AND}{{\rm and }}
\newcommand{\p}{{$p\mspace{1mu}$}}

\newcommand{\loc}{_{\rm loc}}

\newcommand{\eps}{\varepsilon}

\theoremstyle{plain}
\newtheorem{theorem}[equation]{Theorem}
\newtheorem{lemma}[equation]{Lemma}

\numberwithin{equation}{section}

\theoremstyle{definition}

\theoremstyle{remark}
\newtheorem{remark}[equation]{Remark}
\title{Arithmetic three-spheres theorems for quasilinear Riccati type
  inequalities}

\author{Seppo Granlund} \address[S.G.]{University of Helsinki,
  Department of Mathematics and Statistics, P.O. Box 68, FI-00014
  University of Helsinki, Finland} \email{seppo.granlund@pp.inet.fi}

\author{Niko Marola} \address[N.M.]{University of Helsinki, Department
  of Mathematics and Statistics, P.O. Box 68, FI-00014 University of
  Helsinki, Finland} \email{niko.marola@helsinki.fi}

\date{}

\begin{document}



\begin{abstract}
  We consider arithmetic three-spheres inequalities to solutions of
  certain second order quasilinear elliptic differential equations and
  inequalities with a Riccati-type drift term.
\end{abstract}

\maketitle

\section{Introduction and main results}

In this paper we consider the following second order quasilinear
elliptic differential equation with a Riccati-type drift term,
\begin{equation} \label{eq:quasi} -\nabla\cdot \mathcal{A}(x,u,\nabla
  u)+\mathcal{B}(x,u,\nabla u)=0,
\end{equation}
in a domain $G$ of $\R^n$, where $n\geq 2$, and the related pair of
differential inequalities
\begin{align*}
  & -\nabla\cdot \mathcal{A}(x,u,\nabla u)+\mathcal{B}(x,u,\nabla u) \leq 0, \\
  & -\nabla\cdot \mathcal{A}(x,u,\nabla u)+\mathcal{B}(x,u,\nabla u) \geq 0.
\end{align*}
In some sense, solutions to these differential inequalities can be
considered to be sub- and supersolutions to the
equation~\eqref{eq:quasi}. Here $\mathcal{A}$ and $\mathcal{B}$ are
assumed to satisfy certain growth conditions which we specify in
Section~\ref{sect:ests}; see also Section~\ref{sect:Local}. We shall
prove an arithmetic version of the Hadamard three-circles/spheres
theorem for sub- and supersolutions to the equation
\eqref{eq:quasi}. Three-spheres inequalities are central in the
qualitative theory of partial differential equations. As an
application, we obtain a Cauchy--Liouville-type theorem for solutions
to \eqref{eq:quasi} under certain structural assumptions.

The classical Hadamard three-circles theorem for analytic functions
has counterparts for solutions to elliptic differential equations and
inequalities. For instance, suppose that $u$ is a non-constant
$C^2$-smooth $2$-sub\-harmonic function, i.e. $\Delta u \geq 0$, in
the set $\{x\in\R^n:\, |x|< R\}$ and let $M(r)= \sup\{u(x):\, |x|=r\}$
for every $0<r\leq R$. Then by the strong maximum principle $M(r)$ is
a strictly increasing function of $r$. The three-circles theorem, on
the other hand, tells that $M(r)$ is a convex function of $\log r$
when $n=2$, and a convex function of $-r^{2-n}$ when $n>2$; see
Protter--Weinberger~\cite{ProWei}. One application of the theorem is
Liouville's theorem; a uniformly bounded subharmonic function in the
whole $(x,y)$-plane, except possibly at one point, is a constant
\cite{ProWei}.

There are counterparts of the three-circles inequality for solutions
to other elliptic equations. To name a few, we refer to Dow~\cite{Dow},
Landis~\cite{Landis, Landis2}, V\'yborn\'y~\cite{Vyb, Vyb90}, and to a
more recent paper \cite{FrPi} by Fraas and Pinchover and to Miklyukov
et al.~\cite{MRV}, see also the references in these papers.

It is interesting to note that the aforementioned result for
subharmonic functions holds in verbatim also in the case of the
non-linear \p-Laplace equation,
\[
-\nabla\cdot(|\nabla u|^{p-2}\nabla u)=0,
\]
where $1<p<\infty$. This is simply due to the existence of the radial
fundamental solutions $|x|^{(p-n)/(p-1)}$ whenever $1<p<n$, and
$-\log|x|$ for the borderline case $p=n$. Then any function of the
form $a + b|x|^{(p-n)/(p-1)}$, $a,b\in\R$, is a solution to the
\p-Laplace equation in punctured neighborhoods of the origin. Then
the comparison principle gives the following arithmetic version of the
Hadamard three-circles theorem: Suppose that $u$ is a \p-subharmonic
in a domain containing concentric circles of radii $r_1$ and $r_3$ and
the region between them, then for each $r_1<r_2<r_3$
\[
M(r_2) \leq \frac{M(r_1)\log(r_3/r_2)+M(r_3)\log(r_2/r_1)}{\log(r_3/r_1)},
\]
when $p=n$, and in the case in which $1<p<n$ we have
\[
M(r_2) \leq
\frac{M(r_1)\left(r_2^\alpha-r_3^\alpha\right)+M(r_3)\left(r_1^\alpha-r_2^\alpha\right)}{r_1^\alpha-r_3^\alpha},
\]
where $\alpha=(p-n)/(p-1)$. 

Let us also briefly discuss another much stronger form of the
three-spheres inequality for solutions to elliptic equations. In fact,
since there are two different types of three-spheres inequalities, we
shall call the type treated in the present paper {\it arithmetic}
(seen Theorem~\ref{thm:3sphere}) to dissociate it from the other form
of inequalities. Namely, inequalities of the type
\begin{equation} \label{eq:3sphere2}
\|u\|_{B_{r_2}} \leq C\|u\|_{B_{r_1}}^\tau\|u\|_{B_{r_3}}^{1-\tau}, 
\end{equation}
where $0<r_1<r_2<r_3$. Usually, the radius $r_3$ has to be
sufficiently small and $u$ is a $C^2$-solution to an elliptic
differential equation, and $\|u\|$ is the $L^2$ or
$L^\infty$-norm of $u$ on concentric spheres. The constants $C$ and
$\tau\in(0,1)$ depend only on the given elliptic operator and on the
ratios $r_1/r_2$, $r_1/r_3$, and $r_2/r_3$. The classical Hadamard
inequality for analytic functions in an annulus is the inequality
\eqref{eq:3sphere2} with $C=1$ and
$\tau=\log(r_3/r_2)/\log(r_3/r_1)$. For counterparts of the inequality
\eqref{eq:3sphere2} to elliptic equations, see
e.g. Korevaar-Meyers~\cite{KoMe}, Brummelhuis~\cite{Bru}, and Lin et
al. \cite{LNWPisa, LNWRevista}. We also refer to Alessandrini et
al.~\cite{ARRV} for an exhaustive reference list and discussion on the
topic.

In general, three-spheres inequalities of the form \eqref{eq:3sphere2}
do not hold for second order quasilinear elliptic equations of
divergence form. This can be seen from the counterexamples in
\cite{Martio} which concern solutions to the equation \eqref{eq:quasi}
with $\mathcal{B}=0$, $p=n$, and $n\geq 3$ in
\eqref{eq:structuralA}. This phenomenon occurs already in the linear
case as can be deduced from the counterexamples by Pli\'s~\cite{Plis};
in \cite{Plis} the reader can find counterexamples for solutions to
certain second order linear elliptic equations of divergence form with
H\"older continuous coefficients (with any exponent less than one) in
$\R^3$.

In this paper, we are concerned with equations of the general form
\eqref{eq:quasi}, and hence we confine ourselves to the study of an
arithmetic version of the three-spheres theorem. More precisely, we
study the growth of sub- and supersolutions to \eqref{eq:quasi} in
terms of the functions
\[
\mathcal{M}(r) = \esssup\{u(x):\, x\in B_r\}
\]
and
\[
\mathfrak{m}(r) = \essinf\{u(x):\, x\in B_r\},
\]
where $\overline{B}_r\subset G$ is a ball centered at some point in
$G$ and with radius $r$. We discuss the finiteness of these functions
in Remark~\ref{rmk:Finiteness}. As usual, we distinguish three cases
$1<p<n$, $p=n$, and $p>n$ in \eqref{eq:structuralA} and
\eqref{eq:structuralB1}. One of the results in this paper is the
following theorem.

\begin{theorem}[Global arithmetic three-spheres inequality:
  $1<p<n$] \label{thm:3sphere} Suppose that $u$ is a subsolution to
  \eqref{eq:quasi} in $G$ under the structural assumptions
  \eqref{eq:structuralA} and \eqref{eq:structuralB1} for
  $1<p<n$. Assume further that there is a positive number $\tau$ so
  that for every $0<r_1<r_2<r_3$, such that $\overline{B}_{r_3}\subset
  G$, the inequalities $0<\tau\leq r_1/r_2<r_2/r_3<1$ hold. Then
  there exists a constant $0<\lambda<1$, depending only on $n$, $p$,
  $a_0$, $a_1$, $b_1$, and on the ratios $r_1/r_2$, $r_1/r_3$, and
  $r_2/r_3$, for which the inequality
\begin{equation} \label{eq:3sphere}
\mathcal{M}(r_2)\leq \lambda\mathcal{M}(r_1)+(1-\lambda)\mathcal{M}(r_3)
\end{equation}
holds. Also, we obtain the following dual result. If $u$ is a supersolution to
\eqref{eq:quasi} in $G$ the inequality
\begin{equation} \label{eq:3sphereMin} 
\mathfrak{m}(r_2)\geq \lambda\mathfrak{m}(r_1)+(1-\lambda)\mathfrak{m}(r_3)
\end{equation}
holds. The balls $B_{r_1}$, $B_{r_2}$, and $B_{r_3}$ are concentric.
\end{theorem}

We treat the case $p=n$ and $p>n$ in Section~\ref{sect:border} and
Section~\ref{sect:p>n}, respectively.

Compared to the classical three-spheres theorems there are certain
noteworthy differences in our method of proof and also in our
results. We now discuss these aspects in more detail.

Classically, see \cite{ProWei}, the method of proof for three-spheres
theorems is simply based on the usage of radial fundamental solutions
and comparison principle. For quasilinear equations of the form
\eqref{eq:quasi} under the structural assumptions
\eqref{eq:structuralA} and \eqref{eq:structuralB1}, in particular
allowing for dependence on the spatial variable $x$, fundamental
solutions are not available in general. 

Our method, which is based on a priori estimates with respect to solid
balls, compensates the lack of fundamental solutions. However, we are
not able to cover ring shaped domains as in the classical case. Hence,
in Theorem~\ref{thm:3sphere} the growth of sub- and supersolutions is
given in terms of concentric balls instead of concentric spheres.

In our case complexity of the method depends stongly on the structure
parameter $1<p<\infty$. Indeed, as opposed to the case $1<p<n$, when
$p\geq n$ we are able to obtain more explicit formulation for the
convexity/concavity parameter $\lambda$ and, therefore, to attain
Cauchy--Liouville theorems for solutions in the borderline case $p=n$
as in the classical situation.

\medskip

The paper is outlined as follows. After discussing certain preliminary
estimates in Section~\ref{sect:ests}, we shall prove
Theorem~\ref{thm:3sphere} in Section~\ref{sect:proof}. We obtain a
local version of the arithmetic three-spheres inequality for sub- and
supersolutions to the equation which has slightly different structure
as the one in Theorem~\ref{thm:3sphere}, see
Theorem~\ref{thm:3sphereLocal} in Section~\ref{sect:Local}.

In Section~\ref{sect:border} and Section~\ref{sect:p>n} we consider
the borderline case $p=n$ and the case $p>n$ in \eqref{eq:structuralA}
and \eqref{eq:structuralB1}. We obtain a global arithmetic
three-spheres theorem for solutions. However, the method is slightly
different to the one used in connection with Theorem~\ref{thm:3sphere}
in the case $1<p<n$. In the borderline case $p=n$ and in the case
$p>n$, De Giorgi--Ladyzhenskaya--Ural'tseva-type
$L^\infty-L^p$-estimates are replaced by Sobolev's inequality on
spheres as formulated by Gehring and Mostow. By this replacement, we
obtain an explicit formula, up to an absolute constant, for the
parameter $\lambda$ as a function of the ratios $r_1/r_2$, $r_1/r_3$,
and $r_2/r_3$.

In the borderline case $p=n$, as an application of the global
three-spheres inequality with an explicit parameter
$\lambda$, we obtain a Cauchy--Liouville-type result: a bounded entire
solution to \eqref{eq:quasi} under the structure presented in
\eqref{eq:structuralA} and \eqref{eq:structuralB1} must be constant.

\section{Preliminary estimates}
\label{sect:ests}

Let $G$ be a domain in $\R^n$, possibly unbounded, and $n\geq 2$.  For
simplicity and for notational purposes, we assume that $G$ contains
the origin $0$ and that open balls of radius $r$, written as $B_r$,
are centered at $0$. We write the closure of a ball $B_r$ as
$\overline{B}_r$.

The structural conditions \eqref{eq:structuralA} and
\eqref{eq:structuralB1} below ensure that we can consider
\eqref{eq:quasi} in weak form as follows. A function $u\in
W\loc^{1,p}(G)$ is a subsolution [supersolution] of the equation
\eqref{eq:quasi} in $G$ if, and only if, for any relatively compact
$D\subset G$ and any $\eta\in W_0^{1,p}(D)$ with $\eta\geq 0$ in $D$
the inequality
\begin{equation} \label{eq:weakF} \int_D\mathcal{A}(x,u,\nabla
  u)\cdot\nabla\eta\, dx + \int_D \mathcal{B}(x,u,\nabla u)\eta\, dx
  \leq [\geq] 0
\end{equation}
holds; $u$ is a solution in $G$ if the equality holds in
\eqref{eq:weakF} for any relatively compact $D\subset G$ and all
$\eta\in W_0^{1,p}(D)$. Here $\mathcal{A}\colon G\times\R\times\R^n
\to\R^n$ and $\mathcal{B}\colon G\times\R\times\R^n \to \R$ are
assumed to satisfy the Carath\'eodory condition, i.e. each
$A_i(x,t,h)$ ($i=1,\ldots,n$) is measurable in $x\in G$ for every
fixed $(t,h)\in\R\times\R^n$ and continuous in $(t,h)$ for almost
every fixed $x$, and $\mathcal{B}(x,t,h)$ is measurable in $x\in G$
for every fixed $(t,h)\in\R\times\R^n$ and continuous in $(t,h)$ for
almost every fixed $x$. In addition, we shall assume that there are
constants $1<p<n$ and $0<a_0\leq a_1<\infty$ 
such that for all $(t,h)$ in $\R\times\R^n$ and for almost every $x\in
G$ the following structural assumptions apply
\begin{equation} \label{eq:structuralA} \mathcal{A}(x,t,h)\cdot h \geq
  a_0|h|^p, \quad |\mathcal{A}(x,t,h)| \leq a_1|h|^{p-1}.
\end{equation}
For the drift term $\mathcal{B}$ we require it satisfies the following
growth condition: there is a constant
$0<b_1<\infty$ such that for all $(t,h)$ in $\R\times\R^n$ and for
almost every $x\in G$
\begin{equation} \label{eq:structuralB1} |\mathcal{B}(x,t,h)| \leq
  g(x)|h|^{p-1},
\end{equation}
where the function $g:G\to\R$ is defined as
\begin{equation*}
  g(x)=  \left\{\begin{array}{ll}
      b_1 & \textrm{if }\, |x|\leq 1, \\ & \\ 
      \frac{b_1}{|x|} & \textrm{if }\,
      |x|>1.\end{array}\right.
\end{equation*}
By imposing the asymptotic behavior as in \eqref{eq:structuralB1} on
$\mathcal{B}$ we are able to obtain a global three-spheres inequality
in the case of the general form of an elliptic quasilinear
equation. 

We do not assume monotonicity of $\mathcal{A}$ since we do not deal
with existence problems. The prototype operator of $\mathcal{A}$ is
the \p-Laplacian. A growth condition in some ways similar to
\eqref{eq:structuralB1} can be found in \cite{FrPi}, see also Lin et
al.~\cite{LNWRevista} and \cite{LNWPisa}.

We shall use a method exploiting certain convex functions and which
has been previously used to prove oscillation lemmas and H\"older
continuity for solutions to quasilinear elliptic equations
\cite{LaUr}, \cite{DiB}, but which has not been applied to prove
three-spheres inequalities to quasilinear equations with a
Riccati-type drift term \eqref{eq:quasi}.

In what follows, $\phi:\R\to\R$ shall be a convex function satisfying
the following conditions: there exists a sub-interval $I$ of $\R$ such
that
\begin{enumerate}

\item[(C-1)] $\quad \phi\in C^2(I)$; \smallskip

\item[(C-2)] $\quad \phi'(t)^2 \leq \phi''(t) $ for all $t\in I$; \smallskip

\item[(C-3)] $\quad$either (i) $\phi'(t)>0$ or (ii) $\phi'(t)<0$ for
  all $t\in I$.

\end{enumerate}
Positive constants $C$, $C_1$, $C_2$, $C_3$, and $C_4$ depend only on
$n$, $p$, $a_0$, $a_1$, and $b_1$, and may vary from line to
line. Dependence on other parameters than the aforementioned is
written as $C(\eps, \delta)$.

Let us consider some estimates for the composite function $\phi(u)$,
where $u$ is a solution to a differential inequality related to the
equation~\eqref{eq:quasi}.

In regard to the growth condition \eqref{eq:structuralB1} and the
function $g$, we shall next treat only the case in which $|x|>1$ as
the case $|x|\leq 1$ follows easily by modifying the proofs for the
former case; when deriving these a priori estimates in
Lemma~\ref{lemma:1} and Lemma~\ref{lemma:2} and restricting ourselves
to those balls whose radius $r\leq 1$ we can embed the terms caused by
the drift function $\mathcal{B}$ to the corresponding terms produced
by the operator $\mathcal{A}$. See also the discussion in
Section~\ref{sect:Local}.

\begin{lemma} \label{lemma:1} Suppose that $\phi:I\to\R$ is a convex
  function such that the conditions (C-1) through (C-3)(i) listed
  above are valid and that $u$ is a subsolution to \eqref{eq:quasi} in
  $G$, with the structural assumptions \eqref{eq:structuralA} and
  \eqref{eq:structuralB1} for fixed $1<p<n$, such that $u(G)\subset
  I$. There exists a constant $C>0$ such that the inequality
\begin{equation} \label{eq:lemma1}
\int_{B_{(1-\delta)r}}|\nabla \phi(u(x))|^p\, dx \leq \frac{C}{\delta^p}r^{n-p}
\end{equation}
holds for every $\overline{B}_r\subset G$, and $0<\delta<1$.

We obtain an analogous inequality in the case in which $u$ is a
supersolution to \eqref{eq:quasi} and the conditions (C-1) through
(C-3)(ii) are valid for a convex function $\phi:I\to\R$.
\end{lemma}

\begin{proof}
  Suppose that $u$ is a subsolution. Let $\xi\in C_0^\infty(B_r)$ be
  non-negative and $\eta(x) = \phi'(u(x))^{p-1}\xi^p(x)$ at $x\in
  B_r$, here $\phi'(t)>0$ for all $t\in I$. By plugging $\eta$ into
  the inequality \eqref{eq:weakF} and using the structural assumptions
  \eqref{eq:structuralA} and \eqref{eq:structuralB1} we have
\begin{align*}
a_0(p-1)\int_{B_r} & |\nabla u|^p\phi'(u(x))^{p-2}\phi''(u(x))\xi^p\, dx \\
& \leq a_1p \int_{B_r}|\nabla u|^{p-1}\phi'(u(x))^{p-1}\xi^{p-1}|\nabla \xi|\, dx \\
& \qquad \quad + b_1\int_{B_r}|x|^{-1}|\nabla u|^{p-1}\phi'(u(x))^{p-1}\xi^p\, dx.
\end{align*} 
Applying the condition (C-2) on the left-hand side, the preceding
inequality becomes
\begin{align*}
\int_{B_r} & |\nabla u|^p\phi'(u(x))^p\xi^p\, dx \\
& \leq \frac{a_1p}{a_0(p-1)}\int_{B_r}|\nabla u|^{p-1}\phi'(u(x))^{p-1}\xi^{p-1}|\nabla \xi|\, dx \\
& \qquad \quad  + \frac{b_1}{a_0(p-1)}\int_{B_r}|x|^{-1}|\nabla u|^{p-1}\phi'(u(x))^{p-1}\xi^p\, dx.  
\end{align*} 
From Young's inequality with positive $\eps$ we get
\begin{align*}
  \int_{B_r}&|\nabla u|^p\phi'(u(x))^p\xi^p\, dx \\
  & \leq \frac{a_1p\eps}{a_0(p-1)}\int_{B_r}|\nabla u|^p\phi'(u(x))^p\xi^p\, dx +  C_1(\eps)\int_{B_r}|\nabla \xi|^p\, dx \\
  & \qquad  + \frac{b_1\eps}{a_0(p-1)}\int_{B_r}|\nabla u|^p\phi'(u(x))^p\xi^p\, dx + C_2(\eps)\int_{B_r}|x|^{-p}\xi^p\, dx \\
  & \leq C_1\int_{B_r}|\nabla \xi|^p\, dx + C_2r^{n-p},
\end{align*}
where the last inequality is obtained by choosing $\eps$ sufficiently
small so that the first and the third term on the right are absorbed
by the term on the left. The preceding inequality leads to
\eqref{eq:lemma1} by choosing $\xi=1$ on $B_{(1-\delta)r}$, $0\leq
\xi\leq 1$, such that $|\nabla\xi|\leq C/\delta r$ in $B_r$ .

In the case in which $u$ is a supersolution and $\phi'(t)<0$ for all
$t\in I$, at each $x\in B_r$ we set $\eta(x) =
(-\phi'(u(x)))^{p-1}\xi^p(x)$. Then the inequality \eqref{eq:lemma1}
is obtained in the same way as above.
\end{proof}

For any $k\geq 0$, let us write $A_{k,r} = \{x\in B_r:\;
\phi(u(x))>k\}$. We first deduce a Caccioppoli-type estimate on such
level sets of the composite function $\phi(u)$. Then, by referring to
a well known iteration argument, an $L^\infty$--$L^p$ estimate is
obtained.

\begin{lemma} \label{lemma:2} Suppose that $\phi:I\to\R$ is a convex
  function such that the conditions (C-1) through (C-3)(i) listed
  above are valid and that $u$ is a subsolution to \eqref{eq:quasi} in
  $G$ with the structural assumptions \eqref{eq:structuralA} and
  \eqref{eq:structuralB1} such that $u(G)\subset I$. Let
  $\overline{B}_{r_0}\subset G$ and $0<\delta_0<1$ be a fixed
  constant. Assume further that there exists a radius $\hat r>0$ such
  that $\phi(u(x))\leq 0$ at almost every $x\in B_{\hat r}$ and there
  is $\tau$, $0<\tau<1$, such that $\hat r \geq \tau
  (1-\delta_0)r_0$. Then for every $k\geq 0$ and for
  $(1-\delta_0)r_0\leq (1-\delta)r<r\leq r_0$ the inequality
  \begin{equation} \label{eq:lemma2} \int_{A_{k,(1-\delta)r}}|\nabla
    \phi(u(x))|^p\, dx \leq \frac{C(\delta_0,\tau)}{(\delta
      r)^p}\int_{A_{k,r}}(\phi(u(x))-k)^p\, dx
\end{equation}
holds. In addition, the inequality
\begin{equation} \label{eq:lemma2-2} \left(\esssup_{x\in
      B_{(1-\delta_0)r_0}}\phi(u(x))\right)^p \leq
  \frac{C(\delta_0,\tau)}{r_0^n}\int_{A_{0,r_0}}\phi(u(x))^p\, dx
\end{equation}
is valid. All the balls above are concentric.

We obtain analogous inequalities in the case in which $u$ is a
supersolution to \eqref{eq:quasi} and the conditions (C-1) through
(C-3)(ii) are valid for a convex function $\phi:I\to\R$.
\end{lemma}

\begin{proof}
  Suppose that $u$ is a subsolution and let $k\geq 0$ be
  arbitrary. Define $\psi(x) = \max\{\phi(u(x))-k,0\}$ at every $x\in
  B_r$, where $r\leq r_0$ with $\overline{B}_{r_0}\subset G$. Let us
  choose $\eta(x) = \psi(x)\phi'(u(x)))^{p-1}\xi^p(x)$, where $\xi\in
  C_0^\infty(B_r)$ is non-negative. 

  By plugging $\eta$ into the inequality \eqref{eq:weakF}, using the
  structural assumptions \eqref{eq:structuralA} and
  \eqref{eq:structuralB1}, and then (C-2) we have
  \begin{align*} 
    & \int_{B_r}|\nabla u|^p\phi'(u(x))^p\xi^p\, dx  + (p-1) \int_{B_r}|\nabla u|^p\psi(x)\phi'(u(x))^p\xi^p\, dx \nonumber \\
    & \leq \frac{a_1p}{a_0}\int_{B_r}|\nabla u|^{p-1}\psi(x)\phi'(u(x))^{p-1}\xi^{p-1}|\nabla \xi|\, dx \\
    & \qquad \qquad + \frac{b_1}{a_0}\int_{B_r}|x|^{-1}|\nabla
    u|^{p-1}\psi(x)\phi'(u(x))^{p-1}\xi^p\, dx.  \nonumber
\end{align*} 
We drop the second term on the left-hand side and use Young's
inequality with $\eps>0$ to the terms on the right-hand side in the
preceding inequality; it becomes
\begin{align*} 
  & \int_{B_r}|\nabla u|^p\phi'(u(x))^p\xi^p\, dx  \leq  \frac{a_1p}{a_0}\eps\int_{B_r}|\nabla u|^p\phi'(u(x))^p\xi^p\, dx \nonumber \\
  & \quad + C_1(\eps)\int_{B_r}\psi(x)^p|\nabla \xi|^p\, dx + \frac{b_1}{a_0}\eps\int_{B_r}|\nabla u|^p\phi'(u(x))^p\xi^p\, dx \nonumber \\
  & \qquad \qquad + C_2(\eps)\int_{B_r}|x|^{-p}\psi(x)^p\xi^p\, dx.
\end{align*} 
We then choose $\eps$ small enough so that the first and the third
term on the right will be absorbed by the term on the left. Then let
us choose the cut-off $\xi$ so that $\xi( x)=1$ for $x\in
B_{(1-\delta)r}$, $0\leq \xi\leq 1$, and $|\nabla\xi| \leq C/\delta r$
on $B_r$. Furthermore, we clearly have that both $\psi(x)=0$ and
$\nabla\psi(x) = 0$ at a.e. $x\in B_r\setminus A_{k,r}$. Also, by the
hypothesis, $\psi(x) = 0$ on $B_{\hat r}$. Altogether, we have
\[
\int_{A_{k,(1-\delta)r}}|\nabla u|^p\phi'(u(x))^p\, dx \leq
\frac{C}{(\delta
  r)^p}\left(1+\frac{\delta_0^p}{\tau^p(1-\delta_0)^p}\right)\int_{A_{k,r}}\psi(x)^p\,
dx,
\]
and hence the desired inequality \eqref{eq:lemma2}.

\medskip

Suppose $u$ is a supersolution. In this case, let us choose
$\eta(x)=\psi(x)(-\phi'(u(x))))^{p-1}\xi^p(x)$ and we shall proceed as
above to obtain the inequality~\eqref{eq:lemma2}.

\medskip

It is well known that an inequality of the form \eqref{eq:lemma2-2}
follows from an inequality of the type \eqref{eq:lemma2} by a
De Giorgi-type iteration argument, see \cite[Lemma~5.4, page 76]{LaUr}.
\end{proof}

\begin{remark} \label{rmk:Finiteness} It follows by a similar argument
  as in Lemma~\ref{lemma:2} with $\phi(t)=t$ that
  $\mathcal{M}(r)<\infty$ whenever $u$ is a subsolution to
  \eqref{eq:quasi} in $G$ and $\overline{B}_r\subset G$ such that
  $r\leq 1$; in this case the extra hypotheses concerning the
  existence of the ball $B_{\hat r}$ and the constant $\tau$ in
  Lemma~\ref{lemma:2} become obsolete. We may, therefore, conclude by
  a covering argument that $\mathcal{M}(r)<\infty$ for all $r>0$ for
  which $\overline{B}_r\subset G$.

  It can be noted that under our structural assumptions
  \eqref{eq:structuralA} and \eqref{eq:structuralB1} it does not
  necessarily hold that $-u$ is a supersolution to \eqref{eq:quasi}
  whenever $u$ is a subsolution. Hence, to obtain
  \eqref{eq:3sphereMin} we assume that $\mathfrak{m}(r)>-\infty$ for
  all $r>0$ for which $\overline{B}_r\subset G$. This extra hypothesis
  becomes void if we assumed certain homogeneity of the operator
  $\mathcal{B}$. Since this would rule out an interesting set of
  equations, for instance the equation \eqref{eq:Riccati}, we do not
  make such an assumption.
\end{remark}

\section{Theorem~\ref{thm:3sphere}}
\label{sect:proof}

\begin{proof}[Proof of Theorem~\ref{thm:3sphere}]
  We consider first the inequality \eqref{eq:3sphere}. Let $u$ be a
  subsolution and $\eps>0$. We define a convex function $\phi$
  satisfying the conditions (C-1)--(C-3)(i) as
\[
\phi(t) =
-\log\left(\frac{\mathcal{M}(r_3)-t+\eps}{\mathcal{M}(r_3)-\mathcal{M}(r_1)+\eps}\right)
\]
for $t\in(-\infty,\mathcal{M}(r_3)]$. We consider the composite
function $\phi(u(x))$. Define also $\psi(x) =
\max\{\phi(u(x)),0\}$. In what follows, $C$ is always a positive
constant which may vary from line to line and depends only on $n$,
$p$, $a_0$, $a_1$, $b_1$, $r_1/r_2$, $r_1/r_3$, and $r_2/r_3$.

Observe that since $\psi(x)=0$ at each $x\in B_{r_1}$, we have the
Poincar\'e inequality
\[
\int_{B_{(r_2+r_3)/2}}\psi(x)^p\, dx \leq
C\left(\frac{r_2+r_3}{2}\right)^p\int_{B_{(r_2+r_3)/2}}|\nabla\psi(x)|^p\, dx.
\]
We obtain the following $L^\infty$-bound for $\phi(u(x))$ using first
\eqref{eq:lemma2-2} with $r_0 = (r_2+r_3)/2$, $0<\delta_0 =
(r_3-r_2)/(r_2+r_3)<1$, $\hat r = r_1$, and $\tau \leq r_1/r_2$, then
the Poincar\'e inequality, and finally the inequality
\eqref{eq:lemma1},
\begin{align} \label{eq:Linftybound} \left(\esssup_{x\in
      B_{r_2}}\phi(u(x))\right)^p & \leq C\left(\frac{r_2+r_3}{2}\right)^{-n}\int_{A_{0,(r_2+r_3)/2}}\phi(u(x))^p\, dx  \nonumber \\
  & = C\left(\frac{r_2+r_3}{2}\right)^{-n}\int_{B_{(r_2+r_3)/2}}\psi(x)^p\, dx  \nonumber \\
  & \leq
  C\left(\frac{r_2+r_3}{2}\right)^{p-n}\int_{B_{(r_2+r_3)/2}}|\nabla\phi(u(x))|^p\,
  dx \nonumber \\
  & \leq C \nonumber .
\end{align}
The obtained upper bound $C$ is independent of $\eps$. Since $\phi$ is
strictly increasing
\[
\esssup_{x\in B_{r_2}}\phi(u(x)) =
-\log\left(\frac{\mathcal{M}(r_3)-\mathcal{M}(r_2)+\eps}{\mathcal{M}(r_3)-\mathcal{M}(r_1)+\eps}\right),
\]
and further we obtain
\[
\log\left(\frac{\mathcal{M}(r_3)-\mathcal{M}(r_2)+\eps}{\mathcal{M}(r_3)-\mathcal{M}(r_1)+\eps}\right)
\geq -C.
\]
It follows that
\[
\mathcal{M}(r_3)-\mathcal{M}(r_2)+\eps \geq
e^{-C}\left(\mathcal{M}(r_3)-\mathcal{M}(r_1)+\eps\right),
\]
or equivalently
\[
\mathcal{M}(r_2) \leq e^{-C}\mathcal{M}(r_1) + (1-e^{-C})\mathcal{M}(r_3)+(1-e^{-C})\eps,
\]
from which the claim \eqref{eq:3sphere} follows by letting $\eps\to 0$.

\medskip

Suppose that $u$ is a supersolution. To prove the inequality
\eqref{eq:3sphereMin}, let $\eps>0$ and choose a convex function
satisfying the conditions (C-1)--(C-3)(ii) as
\[
\phi(t) =
-\log\left(\frac{t-\mathfrak{m}(r_3)+\eps}{\mathfrak{m}(r_1)-\mathfrak{m}(r_3)+\eps}\right)
\]
for $t\in[\mathfrak{m}(r_3),\infty)$. As above, we consider the
composite function $\phi(u(x))$ and define $\psi(x) =
\max\{\phi(u(x)),0\}$ which also vanishes at each $x\in B_{r_1}$, and
obtain by reasoning as above that for $x\in B_{r_2}$
\[
\phi(u(x)) \leq C,
\]
where the constant $C$ is independent of $\eps$. Hence for each $x\in B_{r_2}$
\[
u(x) + (1-e^{-C})\eps \geq e^{-C}\mathfrak{m}(r_1)+
(1-e^{-C})\mathfrak{m}(r_3).
\]  
We obtain the desired inequality \eqref{eq:3sphereMin} by letting
$\eps\to 0$.
\end{proof}

\section{Local three-spheres theorem}
\label{sect:Local}

In this section, let us consider the equation \eqref{eq:quasi} with
the structural assumptions \eqref{eq:structuralA} and
\begin{equation} \label{eq:structuralB2}
  |\mathcal{B}(x,t,h)| \leq b_1|h|^{p-1}
\end{equation}
for all $(t,h)$ in $\R\times\R^n$ and for almost every $x\in G$, where
$0<b_1<\infty$ and $1<p<n$ are fixed. The prototype equation
becomes
\begin{equation} \label{eq:Riccati}
-\nabla\cdot \mathcal{A}(x,u,\nabla u)= b(x)|\nabla
u|^{p-1},
\end{equation}
where $b:\R^n\to\R$ is a bounded measurable function in $G$. This
equation has been under active consideration, see e.g. \cite{BMMP},
\cite{Martio11}, and \cite{Messano}.

Suppose that $u$ is either a sub- or supersolution to the equation
\eqref{eq:quasi} under the structural conditions
\eqref{eq:structuralA} and \eqref{eq:structuralB2}. Then it is easy to
verify that we are able to obtain an inequality similar to that in
Lemma~\ref{lemma:1} by restricting ourselves to those balls
$\overline{B}_r\subset G$ for which $r\leq 1$. In addition, in this
$r\leq 1$ case the extra hypotheses concerning the existence of the
ball $B_{\hat r}$ and the constant $\tau$ in Lemma~\ref{lemma:2} can
be neglected and again we obtain two estimates similar to those in
Lemma~\ref{lemma:2}. We use the fact that $r\leq 1$ when deriving
these a priori estimates to embed the terms caused by the drift
function $\mathcal{B}$ to the corresponding terms produced by the
operator $\mathcal{A}$.

We omit the proof of the following local result since it resembles
that of Theorem~\ref{thm:3sphere}.

\begin{theorem}[Local arithmetic three-spheres inequality:
  $1<p<n$] \label{thm:3sphereLocal} Suppose that $u$ is a subsolution
  to \eqref{eq:quasi} in $G$ with the structural assumptions
  \eqref{eq:structuralA} and \eqref{eq:structuralB2} for $1<p<n$. For
  every $0<r_1<r_2<r_3\leq 1$, such that $\overline{B}_{r_3}\subset
  G$, there exists a constant $0<\lambda<1$, depending only on $n$,
  $p$, $a_0$, $a_1$, $b_1$, and on the ratios $r_1/r_2$, $r_1/r_3$,
  and $r_2/r_3$, for which the inequality \eqref{eq:3sphere} holds. In
  the case in which $u$ is a supersolution in $G$ the inequality
  \eqref{eq:3sphereMin} is valid.
\end{theorem}

\section{Three-spheres theorem in the borderline case}
\label{sect:border}

In this section, we let $p=n$ and consider solutions to the equation
\eqref{eq:quasi} under the structure \eqref{eq:structuralA} and
\eqref{eq:structuralB1}. This is the so-called borderline case. We
obtain a global arithmetic three-spheres theorem also in this case,
however, the method is slightly different to the previously
presented. De Giorgi--Ladyzhenskaya--Ural'tseva-type estimates are
replaced by an oscillation lemma due to Gehring and Mostow. With this
replacement we are able to obtain an explicit formula for the
convexity parameter $\lambda$ as a function of the ratios $r_1/r_2$,
$r_1/r_3$, and $r_2/r_3$.

As an application of the global arithmetic three-spheres theorem
with an explicit convexity parameter $\lambda$, we obtain a
Cauchy--Liouville-type result in Theorem~\ref{thm:Liouville}.

In what follows, if $A\subset \R^n$ is a non-empty measurable set and
$\sup_{x\in A}|u(x)|<\infty$, we let
\[
\osc_{A}u = \sup_{x\in A}u(x)-\inf_{x\in A}u(x) 
\]
denote the oscillation of $u$ on the set $A$.

We shall need the following variant of Lemma~\ref{lemma:1}.

\begin{lemma} \label{lemma:3} Suppose that $\phi:I\to\R$ is a
  convex function such that the conditions (C-1) through (C-3) listed
  in Section~\ref{sect:ests} are valid and that $u$ is a solution to
  \eqref{eq:quasi} in $G$ with the structural assumptions
  \eqref{eq:structuralA} and \eqref{eq:structuralB1} for $p=n$ such
  that $u(G)\subset I$. Then for every $B_{r_1}\subset B_{r_2}\subset
  \overline{B}_{r_3}\subset G$ the inequality
  \begin{align} \label{eq:lemma3} & \int_{B_{(r_2+r_3)/2}\setminus
      B_{r_2}}|\nabla \phi(u(x))|^n\, dx \\
    & \qquad \leq C\left(\left(\log\frac{2r_3}{r_2+r_3}\right)^{1-n}
      + \left(\log\frac{r_2}{r_1}\right)^{1-n} + \log
      \frac{r_3}{r_1}\right) \nonumber
\end{align}
holds. The constant $C>0$ depends only on $n$, $a_0$, $a_1$, and
$b_1$.
\end{lemma}

\begin{proof}
  The proof is similar to the proof of Lemma~\ref{lemma:1} apart from
  obvious modifications. We obtain, by assuming the condition (C-3)(i)
  and choosing $\eta$ as in the proof of Lemma~\ref{lemma:1},
  \begin{align} \label{eq:ProofLemma3} \int_{B_{r_3}}|\nabla
    u|^n\phi'(u(x))^n\xi^n\, dx & \leq C_1\int_{B_{r_3}}|\nabla
    \xi|^n\,
    dx \\
    & \qquad + C_2\int_{B_{r_3}}|x|^{-n}\xi^n\, dx, \nonumber
\end{align} 
where $C_1$ and $C_2$ are positive constants depending on $n$, $a_0$,
$a_1$, and $b_1$. Let us choose a non-negative cut-off function
$\xi\in C_0^\infty(B_{r_3})$, $0\leq\xi\leq 1$, as follows
\[
\xi(x) = \left\{\begin{array}{ll}
0 & \textrm{if }\, |x|\leq r_1, \\
1 & \textrm{if }\, r_2 <|x| < (r_2+r_3)/2, \\
0 & \textrm{if }\, |x|\geq r_3,
\end{array}\right.
\]
and so that it is admissible for the conformal capacity.
Then the inequality \eqref{eq:ProofLemma3} becomes
\begin{align*}
  \int_{B_{(r_2+r_3)/2}\setminus B_{r_2}} & |\nabla
  u|^n\phi'(u(x))^n\,
  dx 
  \leq C_1\left(\log \frac{2r_3}{r_2+r_3}\right)^{1-n} \\
  & \qquad + C_1\left(\log\frac{r_2}{r_1}\right)^{1-n} + C_2\log
  \frac{r_3}{r_1},
\end{align*}
and hence we obtain the inequality \eqref{eq:lemma3}.
\end{proof}

We have the following global three-spheres inequality with an explicit
formula for the convexity parameter $\lambda$.

\begin{theorem}[Global arithmetic three-spheres inequality:
  $p=n$] \label{thm:3sphereBorder} Suppose that $u$ is a solution to
  \eqref{eq:quasi} in $G$ with the structural assumptions
  \eqref{eq:structuralA} and \eqref{eq:structuralB1} for $p=n$. For
  every $0<r_1<r_2<r_3$, such that $\overline{B}_{r_3}\subset G$,
  there exists a constant $0<\lambda<1$, depending only on $n$, $a_0$,
  $a_1$, $b_1$, and on the ratios $r_1/r_2$, $r_1/r_3$, and $r_2/r_3$,
  for which both the inequality \eqref{eq:3sphere} and
  \eqref{eq:3sphereMin} hold.

  Moreover, in both cases we have the formula
\begin{align} \label{eq:parameter}
  \lambda & =
  \exp\left(-C\left(\left(\log\frac{2r_3}{r_2+r_3}\right)^{1-n} +
      \left(\log\frac{r_2}{r_1}\right)^{1-n} \right. \right. \\
  & \qquad \qquad \qquad \left. \left.+ \log
      \frac{r_3}{r_1}\right)^{1/n}\left(\log\frac{r_2+r_3}{2r_2}\right)^{-1/n}\right),
  \nonumber
\end{align}
where $C>0$ is an absolute constant depending only on $n$, $a_0$, $a_1$, and $b_1$.
\end{theorem}

\begin{proof}
  We first show that $u$ is monotone in the sense of Lebesgue,
  i.e. that $u$ reaches its extrema on the boundary of any relatively
  compact subdomain $D$ of $G$. We can consider only the maximum
  principle for $u$ as the minimum principle is treated
  similarly. Assume, on the contrary, that $u(x_0)=\max_{x\in
    D}u(x)>\max_{x\in\partial D}u(x)$ for some $x_0\in D$. It follows
  that $L=u(x_0)=\max_{x\in B_\rho(x_0)}u(x)$ for some
  $\overline{B}_\rho(x_0)\subset D$, where $\rho\leq 1$. Then the
  function $v=L-u$ is non-negative in $B_\rho$, and $v(x_0)=0$, and
  also $v$ satisfies an equation similar to \eqref{eq:quasi} with
  analogous structure conditions. Since the Harnack inequality holds
  for the solution $v$ \cite[Theorem 1.1]{Tru}, it follows that $u=L$
  on $B_\rho(x_0)$. It is now easy to see that the set where $u=L$ can
  be expanded to be the whole $D$ which, in turn, leads to a
  contradiction.

  Let us now turn to the proof of the inequality \eqref{eq:3sphere} by
  considering the increasing function $\phi(t)$ as defined in the
  proof of Theorem~\ref{thm:3sphere} for
  $t\in(-\infty,\mathcal{M}(r_3)]$. We consider the composite function
  $\phi\circ u$ which is both monotone and continuous since $u$ can be
  shown to be locally H\"older continuous \cite{Tru}. Then Sobolev's
  inequality on spheres as formulated by Gehring \cite{G} and Mostow
  \cite[Lemma~4.3]{Mostow} applied to $\phi\circ u$ gives the
  inequality
\[
\int_{r_2}^{(r_2+r_3)/2}\left(\osc_{\partial
    B_t}\phi(u(x))\right)^n\frac{dt}{t} \leq C
\int_{B_{(r_2+r_3)/2}\setminus B_{r_2}}|\nabla\phi(u(x))|^n\, dx,
\]
where $C>0$ is a constant depending only on $n$. We have by monotonicity
\begin{equation} \label{eq:GM} \left(\osc_{B_{r_2}}\phi(u)\right)^n
  \leq \frac{C}{\log((r_2+r_3)/2r_2)}\int_{B_{(r_2+r_3)/2}\setminus
    B_{r_2}}|\nabla \phi(u(x))|^n\, dx.
\end{equation}
Recall that $\sup_{x\in B_{r_1}}\phi(u(x))=0$. Now, by combining
\eqref{eq:GM} with the inequality \eqref{eq:lemma3} we have
\begin{align*}
  \sup_{x\in B_{r_2}}\phi(u(x)) & = \sup_{x\in
    B_{r_2}}\phi(u(x)) -\sup_{x\in B_{r_1}}\phi(u(x)) \leq \osc_{B_{r_2}}\phi(u) \\
  & \leq C\left(\log
    \frac{r_2+r_3}{2r_2}\right)^{-1/n}\left(\left(\log\frac{2r_3}{r_2+r_3}\right)^{1-n}
  \right. \\
  & \qquad \qquad \left. + \left(\log\frac{r_2}{r_1}\right)^{1-n} +
    \log \frac{r_3}{r_2}\right)^{1/n}.
\end{align*}
Then as in the proof of Theorem~\ref{thm:3sphere}, we obtain the
inequality \eqref{eq:3sphere}, and analogously the inequality
\eqref{eq:3sphereMin}, with the explicit convexity parameter
$\lambda$.
\end{proof}

Let us state an application of the global arithmetic three-spheres
inequality obtained in Theorem~\ref{thm:3sphereBorder}. We obtain the
following Cauchy--Liouville-type result.

\begin{theorem} \label{thm:Liouville}
Suppose that $u$ is a bounded solution to \eqref{eq:quasi} in $\R^n$
with the structural assumptions \eqref{eq:structuralA} and
\eqref{eq:structuralB1} with $p=n$. Then $u$ is constant.
\end{theorem}

\begin{proof}
  Suppose, on the contrary, that $u$ is a non-constant bounded
  solution in $\R^n$. We may assume that $u$ is non-negative as $u+a$,
  $a>0$, satisfies an equation similar to \eqref{eq:quasi}. By the
  hypothesis, there exists a constant $M$ such that $0\leq u(x)\leq M$
  at every $x\in B_r$, $r>0$; in addition, we can assume that
  $\inf_{x\in\R^n} u(x)=0$. Then by letting $r_3\to\infty$ the
  convexity parameter $\lambda$ in \eqref{eq:parameter} of
  Theorem~\ref{thm:3sphereBorder} tends to some number in $(0,1)$,
  written as $\lambda_\infty$. Indeed, it is straightforward to verify
  that $\lambda_\infty=\exp(-C)$, where $C>0$ is from
  \eqref{eq:parameter} and depends only on $n$, $a_0$, $a_1$, and
  $b_1$. Moreover, the inequality \eqref{eq:3sphere} becomes
  \begin{equation} \label{eq:LiouvilleP}
  \mathcal{M}(r_2)-(1-\lambda_\infty)M \leq
  \lambda_\infty\mathcal{M}(r_1)
  \end{equation}
  and it holds for every $0<r_1<r_2<\infty$. 

  We can choose $r_1$ in such a way that $ \mathcal{M}(r_1)$ becomes
  arbitrarily small, possibly by transferring the origin to some point
  $x_0$. It is easy to check that in this new coordinate system we are
  able to obtain a growth condition similar to the condition
  \eqref{eq:structuralB1} with a function $g(x-x_0)$ and the constant
  $b_1$ depending now on $x_0$.

  We then may choose $r_2$ so that $ \mathcal{M}(r_2)$ is close to
  $M$; choose $r_2$ so that $\mathcal{M}(r_2) = \Theta M$, where
  $\Theta > (1-\lambda_\infty)$. Hence we reach a contradiction in
  \eqref{eq:LiouvilleP}.
\end{proof}

We remark that the preceding Cauchy--Liouville-type result
for solutions to \eqref{eq:quasi} cannot be obtained directly from the
Harnack inequality since for general equations involving a drift term
$\mathcal{B}$ a constant $C$ in the Harnack inequality depends on the
radius of a ball \cite[Theorem 1.1]{Tru}.

\subsection{$\mathcal{A}$-harmonic equation}

Let us consider in passing solutions to the
$\mathcal{A}$-harmonic equation
\[
-\nabla\cdot \mathcal{A}(x,u,\nabla u)=0
\]
with the structural conditions \eqref{eq:structuralA} with $p=n$.

Suppose that $u$ is a solution to the $\mathcal{A}$-harmonic equation
in $G$, and that the conditions (C-1)--(C-3) are valid for a convex
function $\phi:I\to\R$. As in the proof of Lemma~\ref{lemma:3}, by
choosing a non-negative cut-off function $\xi\in C_0^\infty(B_{r_3})$,
$0\leq\xi\leq 1$, such that
\[
\xi(x) = \left\{\begin{array}{ll}
1 & \textrm{if }\, |x| < (r_2+r_3)/2, \\
0 & \textrm{if }\, |x| \geq r_3,
\end{array}\right.
\]
and so that it is admissible for the conformal capacity, we have that
for every $B_{r_1}\subset B_{r_2}\subset\overline{B}_{r_3}\subset G$
the inequality
  \begin{equation*} 
    \int_{B_{(r_2+r_3)/2}\setminus
      B_{r_2}}|\nabla \phi(u(x))|^n\, dx \leq C\left(\log\frac{2r_3}{r_2+r_3}\right)^{1-n}
  \end{equation*}
  holds with a constant $C>0$ depending only on $n$, $a_0$, and
  $a_1$. Then as in the proof of Theorem~\ref{thm:3sphereBorder} it is
  straightforward to verify that for every $0<r_1<r_2<r_3$, such that
  $\overline{B}_{r_3}\subset G$, there exists a constant
  $0<\lambda<1$, depending only on $n$, $a_0$, $a_1$, $r_1/r_2$,
  $r_1/r_3$, and $r_2/r_3$, for which both the inequality
  \eqref{eq:3sphere} and \eqref{eq:3sphereMin} hold. Moreover, the
  convexity parameter becomes
\[
\lambda =
\exp\left(-C\left(\log\frac{2r_3}{r_2+r_3}\right)^{(1-n)/n}\left(\log\frac{r_2+r_3}{2r_2}\right)^{-1/n}\right),
\]
where the constant $C>0$ depends only on $n$, $a_0$, and $a_1$.

Notice that if $u$ is a bounded solution to the $\mathcal{A}$-harmonic
equation in $\R^n$, then by letting $r_3\to \infty$ in
\eqref{eq:3sphere} it can be seen that $\mathcal{M}(r_2)\leq
\mathcal{M}(r_1)$ for all $r_1<r_2$, and moreover that $u$ must be
constant.
 
Analogous results in the case $p>n$ in \eqref{eq:structuralA} can be
also obtained as will be shown in Section~\ref{sect:p>n}.

\section{Three-spheres theorem in the case $p>n$}
\label{sect:p>n}

We close this paper by treating the case $p>n$ in
\eqref{eq:structuralA} and \eqref{eq:structuralB1} in the following
theorem.

\begin{theorem}[Global arithmetic three-spheres inequality:
  $p>n$] \label{thm:3sphere-p>n} Suppose that $u$ is a solution to
  \eqref{eq:quasi} in $G$ with the structural assumptions
  \eqref{eq:structuralA} and \eqref{eq:structuralB1} for $p>n$. For
  every $0<r_1<r_2<r_3$, such that $\overline{B}_{r_3}\subset G$,
  there exists a constant $0<\lambda<1$, depending only on $p$, $n$,
  $a_0$, $a_1$, $b_1$, and on the ratios $r_1/r_2$, $r_1/r_3$, and
  $r_2/r_3$, for which both the inequality \eqref{eq:3sphere} and
  \eqref{eq:3sphereMin} hold. Moreover, the parameter $\lambda$ becomes
\[
\lambda =
\exp\left(-C\Lambda\left(\frac{r_1}{r_2},\frac{r_1}{r_3},\frac{r_2}{r_3}\right)\right),
\]
where the function $\Lambda(r_1/r_2,r_1/r_3,r_2/r_3)$ is as in
\eqref{eq:Lambda} below and $C>0$ is an absolute constant depending
only on $p$, $n$, $a_0$, $a_1$, and $b_1$.
\end{theorem}

\begin{proof}
  The method of proof is as earlier. Let
  $\phi:(-\infty,\mathcal{M}(r_3)]\to\R$ be a convex function as
  defined in the proof of Theorem~\ref{thm:3sphere}. As in
  Section~\ref{sect:border}, we use Sobolev's inequality on spheres
  (see e.g. \cite[Lemma~2.10]{MZ}) and monotonicity to obtain
\begin{align*}
  \left(\osc_{B_{r_2}}\phi(u)\right)^p & \log\left(\frac{r_2+r_3}{2r_2}\right) \leq
  \int_{r_2}^{(r_2+r_3)/2}\left(\osc_{\partial
      B_t}\phi(u)\right)^p\frac{dt}{t} \\
  & \qquad \leq C\int_{B_{r_2}}^{(r_2+r_3)/2}\left(t^{p-n}\int_{\partial B_t}|\nabla\phi(u(x))|^p\, dS\right)\, dt \\
  & \qquad \leq
  C\left(\frac{r_2+r_3}{2}\right)^{p-n}\int_{B_{(r_2+r_3)/2}\setminus
    B_{r_2}}|\nabla\phi(u(x))|^p\, dx.
\end{align*}
By a straightforward modification of Lemma~\ref{lemma:3}, we have for
every $B_{r_1}\subset B_{r_2}\subset \overline{B}_{r_3}\subset G$
\begin{align*} \label{eq:p>n} \int_{B_{(r_2+r_3)/2}\setminus B_{r_2}} & |\nabla
  \phi(u(x))|^p\, dx \leq C\left(\Capc_p(B_{(r_2+r_3)/2},B_{r_3})\right. \\
  & \qquad + \left.\Capc_p(B_{r_1},B_{r_2}) + r_1^{n-p}-r_3^{n-p}\right),
\end{align*}
where $C>0$ is an absolute constant depending only on $p$, $n$, $a_0$,
$a_1$, and $b_1$. Here $\Capc_p(B_r,B_R)$ denotes the (variational)
\p-capacity of a condenser $(B_r,B_R)$ for which
\[
\Capc_p(B_r,B_R) =
  C\left(R^{(p-n)/(p-1)}-r^{(p-n)/(p-1)}\right)^{1-p},
\]
whenever $0<r<R$ \cite[Section~2.11]{HKM}, and where $C$ is a constant
depending on $p$ and $n$ only. Hence,
\begin{align} 
  & \sup_{x\in B_{r_2}}\phi(u(x)) \leq  C\left(\log\frac{r_2+r_3}{2r_2}\right)^{-1/p}\left(\frac{r_2+r_3}{2}\right)^{1-n/p} \nonumber \\
  & \quad \cdot \biggl\{\Capc_p(B_{(r_2+r_3)/2},B_{r_3}) +
    \Capc_p(B_{r_1},B_{r_2}) +
    r_1^{n-p}-r_3^{n-p}\biggr\}^{1/p} \nonumber \\
  & \quad =:
  C\Lambda\left(\frac{r_1}{r_2},\frac{r_1}{r_3},\frac{r_2}{r_3}\right) \label{eq:Lambda},
\end{align}
where $C>0$ is an absolute constant depending only on $p$, $n$, $a_0$,
$a_1$, and $b_1$. The claim follows.
\end{proof}

By a similar method as in sections \ref{sect:ests} and
\ref{sect:proof}, it is possible to study arithmetic three-spheres
theorems also for sub- and supersolutions in the case $p>n$ as
presented in Theorem~\ref{thm:3sphere} for $1<p<n$. Also, and as in
Theorem~\ref{thm:3sphere}, this approach would give implicit dependence on
the ratios $r_1/r_2$, $r_1/r_3$, and $r_2/r_3$ for the parameter
$\lambda$. However, we do not pursue this approach in this paper.

\end{document}